\documentclass[11pt,letterpaper]{article}
\newif\ifreviewnotes
\reviewnotesfalse
\usepackage[T1]{fontenc}
\usepackage[utf8]{inputenc}
\usepackage{lmodern,microtype}
\usepackage[margin=1in]{geometry}
\usepackage{amsmath,amssymb,amsthm,mathtools}
\usepackage{graphicx,booktabs,tabularx,array,enumitem,xcolor}
\usepackage[font=small,labelfont=bf]{caption}
\usepackage[authoryear,round]{natbib}
\usepackage[colorlinks=true,linkcolor=blue!45!black,citecolor=blue!45!black,urlcolor=blue!45!black]{hyperref}
\usepackage{fancyhdr,needspace,flafter,float,placeins}

\setlist{topsep=4pt,itemsep=3pt,parsep=0pt}
\newtheorem{theorem}{Theorem}[section]
\newtheorem{proposition}[theorem]{Proposition}
\newtheorem{lemma}[theorem]{Lemma}
\newtheorem{corollary}[theorem]{Corollary}
\theoremstyle{remark}
\newtheorem{remark}[theorem]{Remark}
\newcommand{\R}{\mathbb R}
\newcommand{\one}{\mathbf 1}
\newcommand{\norm}[1]{\left\lVert#1\right\rVert}
\newcommand{\diag}{\operatorname{diag}}
\newcommand{\Lip}{\operatorname{Lip}}
\newcommand{\eps}{\varepsilon}

\newcommand{\Real}{\mathbb R}

\DeclareMathOperator{\rank}{rank}
\DeclareMathOperator{\cof}{cof}
\title{\bfseries Convex Equilibrium Potentials for Monge--Amp\`ere:\\
Curvature Bounds and A Posteriori Evaluation}
\author{%
  \begin{minipage}[t]{0.30\textwidth}
    \centering
    \normalsize Edison V\'azquez\textsuperscript{\(\dagger\)} \\[0.4em]
    \footnotesize Department of Mathematics\\
    CINVESTAV\\[0.3em]
    \href{mailto:evazquez@math.cinvestav.mx}{\texttt{evazquez@math.cinvestav.mx}}
  \end{minipage}%
  \hspace{0.015\textwidth}%
  \begin{minipage}[t]{0.30\textwidth}
    \centering
    \normalsize Yingdong Lu\\[0.4em]
    \footnotesize IBM Research\\[0.3em]
    \href{mailto:yingdong@us.ibm.com}{\texttt{yingdong@us.ibm.com}}
  \end{minipage}%
  \hspace{0.015\textwidth}%
  \begin{minipage}[t]{0.30\textwidth}
    \centering
    \normalsize Tomasz Nowicki\\[0.4em]
    \footnotesize IBM Research\\[0.3em]
    \href{mailto:tnowicki@us.ibm.com}{\texttt{tnowicki@us.ibm.com}}
  \end{minipage}%
}
\date{}
\hypersetup{pdftitle={Convex Equilibrium Potentials for Monge-Ampere: Curvature Bounds and A Posteriori Evaluation},pdfauthor={Edison Vazquez, Yingdong Lu, Tomasz Nowicki},pdfsubject={Convex equilibrium networks and a posteriori error bounds},pdfkeywords={equilibrium networks, convexity, Monge-Ampere, curvature, a posteriori bounds}}
\begin{document}
\ifreviewnotes

\hypersetup{pageanchor=true}
\fi
\setcounter{page}{1}
\maketitle
\begingroup
\renewcommand{\thefootnote}{\fnsymbol{footnote}}
\footnotetext[2]{corresponding author.}
\endgroup
\begin{abstract}
Equilibrium neural networks define their hidden state as the fixed point of a contraction, so their effective depth is set at evaluation time rather than by the architecture. Using them for fully nonlinear PDEs raises two issues that explicit networks avoid: second derivatives may grow with the effective depth, and the potential is only ever evaluated approximately. We study nonnegative softplus equilibrium networks with contraction factor $q<1$, a convex class suited to the Monge--Amp\`ere equation. A positivity cancellation in the resolvent shows that the Hessian grows at most like $(1-q)^{-2}$, rather than the $(1-q)^{-3}$ of a direct estimate, and an explicit one-neuron family shows that this exponent is sharp; third derivatives satisfy a sharp $(1-q)^{-3}$ bound. Residual-based a posteriori estimates then control the state, its input derivatives and, in two dimensions, the Hessian determinant at an approximate equilibrium. Combined with interval arithmetic on complete covers of the domain, they yield rigorous uniform bounds on the Monge--Amp\`ere residual and, through a comparison argument, on the solution error of a fixed trained network. On two manufactured Dirichlet problems on the unit disk, all six trained models are certified with residual below $0.25$ and boundary error below $0.01$. Adaptive stopping halves float64 evaluation time at best, whereas certified adaptive evaluation remains costlier than fixed depth.
\end{abstract}

\section{Introduction}
\label{sec:intro}

Let $\Omega\subset\Real^n$ be a bounded domain, and let $f:\Omega\to[0,\infty)$ and $g:\partial\Omega\to\Real$ be given. The Dirichlet problem for the Monge--Amp\`ere equation asks for a convex function $u^*:\Omega\to\Real$ such that
\begin{equation}\label{eq:pde}
\det D^2u^*=f\quad\text{in }\Omega,\qquad u^*=g\quad\text{on }\partial\Omega.
\end{equation}
Convexity is not a technicality here; it is part of what a solution means. Without it the equation is not elliptic, uniqueness fails, and the comparison principle on which the theory rests is lost \citep{gutierrez2016}. A numerical method must therefore either enforce convexity or recover it in the limit.

Neural networks can enforce it by construction. An input convex neural network \citep{amos2017} composes nonnegative linear maps with convex, nondecreasing activations, so every member of the class is convex in its input. \citet{nystrom2023} train such networks on the residual of \eqref{eq:pde} and estimate their error through the stability of the equation. We keep the same principle but let the depth of the network be decided by a fixed-point iteration rather than by the architecture. For each input $x$, with the parameters held fixed, we iterate
\[
z_{\ell+1}(x)=\phi(Wz_\ell(x)+Ax+b),\qquad z_0(x)=0,
\]
where $\phi(t)=\log(1+e^t)$ is the softplus function applied componentwise. Each step is a layer of an input convex network, and all layers share the parameters $(W,A,b)$. When $W\ge0$ and $\norm{W}_\infty\le q$ for some real number $0<q<1$, the iteration is a contraction and converges to a unique equilibrium $z^*(x)$, which defines the potential
\begin{equation}\label{eq:model}
\begin{aligned}
z^*(x) &= \phi(Wz^*(x)+Ax+b),\\
u_\theta(x) &= w^\top z^*(x)+a^\top x+c,
\end{aligned}
\end{equation}
where $\theta$ collects the parameters which will be formally defined in~\eqref{eqn:parameters}. With $w\ge0$, the potential $u_\theta$ is convex as the limit of convex iterates.

Such an equilibrium network is a weight-tied network of unbounded depth. Its parameter count does not depend on depth, and the number of iterations can be chosen at evaluation time, point by point, to trade accuracy against cost. The contraction gap $\delta=1-q$ sets the effective depth: the iteration error decays like $q^\ell$, so each additional digit of accuracy costs about $\delta^{-1}\log10$ iterations.

This flexibility has a price, and the Monge--Amp\`ere equation exposes it. The operator is the determinant of the Hessian, so whatever happens to second derivatives is amplified in the residual. Two things can go wrong. The first is curvature. As $q\to1$ the network becomes effectively deeper, and a direct estimate allows its Hessian to grow like $\delta^{-3}$, precisely in the regime where the implicit model differs most from a shallow explicit one. The second is evaluation. The equilibrium is only ever computed approximately, so a determinant evaluated at a computed point describes an approximation of the network rather than the network itself. Neither issue arises for an explicit network, and neither is settled by convexity and contraction alone.

It is useful to separate three errors. The \emph{evaluation error} compares a numerical evaluation with the exact equilibrium potential at fixed parameters. The \emph{PDE residual} compares the exact determinant $\det D^2u_\theta$ with the prescribed right-hand side $f$. The \emph{solution error} compares $u_\theta$ with $u^*$. Reducing the first does not eliminate the second, and a pointwise bound on the second does not control the third uniformly. A guarantee about $u_\theta$ as an approximation of $u^*$ must pass through all three.

This paper follows that chain. We first show that curvature is better behaved than the direct estimate suggests. A positivity cancellation,
\begin{equation}\label{eq:introcancellation}
(I-\diag(\phi'(s))W)^{-1}\phi''(s)\le\one,
\end{equation}
reflects the fact that a softplus unit with slope close to one has small intrinsic curvature. It yields a Hessian bound of order $\delta^{-2}$, and an explicit one-neuron family shows that this exponent cannot be improved. Third derivatives obey a sharp bound of order $\delta^{-3}$, and positive weighted norms extend the analysis to recurrent matrices with spectral radius below one. We then turn to evaluation. Residual-based a posteriori estimates bound the error in the state, in its first and second input derivatives and, in two dimensions, in the Hessian determinant, at any approximate equilibrium. They apply whether the derivatives come from a finite unroll or from approximate linear solves, and they account for the perturbation of the coefficients themselves caused by the state error. Finally, we pass from points to the whole domain. Interval arithmetic on complete covers of the disk and its boundary gives rigorous uniform bounds on the Monge--Amp\`ere residual, and a comparison argument converts them into a bound on the solution error.

We test this chain on two manufactured Dirichlet problems on the unit disk, with three trained networks each. All six networks are certified with residual below $0.25$ and boundary error below $0.01$. The experiments also measure how much adaptive evaluation saves and how much certification costs.

\subsection{Related work}
\paragraph{Convex neural representations.}
Input convex neural networks (ICNNs) enforce convexity through nonnegative hidden weights and convex, nondecreasing activations \citep{amos2017}. Their construction follows standard composition rules for convex functions \citep{boyd2004}. Convex neural potentials also arise in optimal transport: \citet{makkuva2020} learn transport maps through gradients of ICNN potentials. These applications motivate control of derivatives as well as function values. Our model retains the positivity structure of an ICNN but defines its hidden state implicitly. Convexity follows from the established ICNN argument applied to successive fixed-point iterates; the curvature estimates below address the additional effect of the equilibrium resolvent.

\paragraph{Equilibrium models and implicit differentiation.}
Deep equilibrium models replace a finite stack of layers by a fixed-point equation and differentiate through its equilibrium \citep{bai2019}. Monotone operator equilibrium networks provide a different route to well-posedness and convergent equilibrium solvers \citep{winston2020}. Non-Euclidean contraction analysis gives further conditions for existence, uniqueness and robustness of implicit networks \citep{jafarpour2021}. Well-posedness conditions of Perron--Frobenius type for implicit models are given by \citet{elghaoui2021}. Our nonnegative recurrent matrices admit a particularly simple positive weighted-norm formulation. For input derivatives, \citet[Proposition 1, equations (8)--(9)]{huang2024} derive first- and second-order identities for a general contractive fixed-point map. The identities used here specialize that analysis. Our a posteriori estimates additionally account for the perturbation of their coefficients when the equilibrium state is computed only approximately.

\paragraph{Monge--Amp\`ere approximation.}
The analytical theory of the Monge--Amp\`ere equation treats convexity as part of the solution concept and supplies comparison and stability tools \citep{gutierrez2016}. In numerical analysis, \citet{froese2011} construct monotone finite-difference schemes converging to viscosity solutions. Neural approaches provide a different approximation framework. Physics-informed neural networks and the deep Galerkin method train neural representations using differential equations and boundary or initial conditions \citep{raissi2019,sirignano2018}. For the Dirichlet Monge--Amp\`ere problem specifically, \citet{nystrom2023} study ICNN approximations, derivative estimates and stability. Our experiments use a loss of the same form, with an equilibrium representation.

\paragraph{Curvature and error certification.}
\citet{entesari2024} derive compositional curvature bounds for explicit neural networks, and \citet[Lemma 6.3]{nystrom2023} bound derivatives up to third order for explicit softplus ICNNs. For neural PDE approximations, \citet{mishra2023} relate generalization error to training error and sampling through stability properties of the underlying equation, while \citet{eiras2024} develop regional error certification for physics-informed neural networks. We distinguish the error of evaluating an implicit potential from its PDE residual and its error relative to the PDE solution. The positive softplus resolvent yields the cancellation responsible for the inverse-quadratic curvature bound; residual estimates then propagate state and derivative errors to the determinant. Ball arithmetic \citep{johansson2017} encloses rounding effects in the regional implementation.

\subsection{Organization and scope}
Sections~\ref{sec:prelim}--\ref{sec:divhession} introduce the equilibrium representation and the derivative identities needed for the analysis. Section~\ref{sec:curvature} proves the curvature cancellation, the inverse-quadratic bound and its optimality, and states the sharp third-derivative extension. Section~\ref{sec:certificate} develops the approximate-state certificate; Section~\ref{sec:regional} extends it to the full domain and to solution-error bounds. Sections~\ref{sec:procedure}--\ref{sec:numerics} describe adaptive evaluation and the six-model study. The appendices contain supporting proofs, the weighted-norm and local-envelope extensions, and finite-unroll derivative recurrences. The guarantees concern evaluation of fixed trained models; they do not establish convergence of training or superiority to explicit ICNNs.

\section{Assumptions, notation and the equilibrium}
\label{sec:prelim}

The representation analysis applies to $x\in\R^n$ with $m$ hidden units representing $z^*$. The determinant certificate and the experiments are stated for $n=2$. Parameters have sizes
\begin{align}
\label{eqn:parameters}
W\in\R^{m\times m},\quad A\in\R^{m\times n},\quad
b,w\in\R^m,\quad a\in\R^n,\quad c\in\R.
\end{align}
Throughout, parameters are fixed when differentiating in $x$, and there exists a real number $q\in (0,1)$ such that 
\begin{equation}\tag{H}\label{eq:hypotheses}
W\ge0,\qquad w\ge0,\qquad \norm{W}_\infty\le q.
\end{equation}
Inequalities between vectors or matrices are entrywise unless stated otherwise. The norm $\norm{M}_\infty$ is the maximum absolute row sum; $\norm{M}_2$ is the spectral norm. Set
\begin{equation}\label{eq:constants}
\delta=1-q,\qquad K_A:=\max_{1\le r\le m}\norm{A_{r,:}}_2.
\end{equation}
The input matrix and affine readout are unrestricted in sign. The Hadamard product (or the entrywise product) is denoted by $\odot$, and powers of vectors are also defined entrywisely.

Softplus $\phi$ is smooth, convex and $1$-Lipschitz, with $0<\phi'<1$ and $0<\phi''\le1/4$. Under \eqref{eq:hypotheses}, the state map~\eqref{eq:model} is a contraction. It has a unique equilibrium, and the ICNN composition argument implies convexity of the equilibrium potential $u_{\theta}$. For an approximate state $\hat z$, we have,
\[
\norm{z^*(x)-\hat z}_\infty\le\rho_z/\delta,
\qquad \rho_z=\norm{\hat z-\phi(W\hat z+Ax+b)}_\infty.
\]
The precise statements and proofs are given in Appendix~\ref{app:foundations}.

\section{Implicit derivatives and positive Hessian structure}
\label{sec:divhession}
At the exact equilibrium define
\begin{equation}\label{eq:notations}
\begin{gathered}
s=Wz^*+Ax+b,\quad
\gamma=\phi'(s)\odot(\one-\phi'(s)), \quad \Psi=\diag(\phi'(s)),\\
B=I-\Psi W,\qquad S=B^{-1},\qquad V=(I-W\Psi)^{-1}A.
\end{gathered}
\end{equation}
Write $A_{:,i}$ for column $i$ of $A$, $J_i=\frac{\partial}{\partial x_i} z^*(x)$, $K_{ij}=\frac{\partial^2}{\partial x_i \partial x_j}z^*(x)$ and $J=[J_1,\ldots,J_n]$. Note that the order of $\Psi W$ and $W\Psi$ in \eqref{eq:notations} matters.

The equilibrium is smooth, and implicit differentiation gives
\[
BJ_i=\Psi A_{:,i},\qquad BK_{ij}=\gamma\odot v_i\odot v_j,
\qquad v_i=A_{:,i}+WJ_i=V_{:,i}.
\]
The resolvent satisfies $S\ge I$ and $\norm S_\infty\le\delta^{-1}$. In particular, with $\beta=(S^\top w)\odot\gamma$,
\[
H:=D^2u_\theta=V^\top\diag(\beta)V\succeq0,
\qquad D^2z^*[h,h]=S[\gamma\odot(Vh)^2].
\]
These identities expose the positive structure used in the curvature argument. Appendix~\ref{app:foundations} gives the derivative derivations, the strict-convexity criterion and the two-dimensional determinant identity.

\section{A curvature bound with an optimal gap exponent}\label{sec:curvature}
\begin{lemma}[Resolvent--curvature cancellation]\label{lem:cancel}
Under assumption~\eqref{eq:hypotheses},
\begin{equation}\label{eq:cancel}
0\le S\gamma\le\one.
\end{equation}
\end{lemma}
\begin{proof}
Since $W\one\le q\one\le\one$, we have
\[
0\le\gamma=\phi'(s)\odot(\one-\phi'(s))\le\one-\phi'(s)
\le\one-\Psi W\one=B\one.
\]
Multiplication by the nonnegative matrix $S$ preserves entrywise inequalities and gives \eqref{eq:cancel}.
\end{proof}

\begin{theorem}[Uniform curvature bound]\label{thm:curvature}
For every $x\in\R^n$ under assumption~\eqref{eq:hypotheses},
\begin{equation}\label{eq:curvature}
\norm{D^2u_\theta(x)}_2\le C_2,
\qquad C_2=\norm{w}_1K_A^2
\min\left\{\delta^{-2},\frac1{4\delta^3}\right\}.
\end{equation}
The bound is uniform in the input and in the bias $b$. Width enters only through the displayed parameter norms.
\end{theorem}
\begin{proof}
For a Euclidean unit vector $h$, $\norm{Ah}_\infty\le K_A$. Since $\norm{(I-W\Psi)^{-1}}_\infty\le\delta^{-1}$,
\[
\norm{Vh}_\infty\le K_A/\delta.
\]
Positivity and Lemma~\ref{lem:cancel} imply
\[
0\le D^2z^*[h,h]
=S[\gamma\odot(Vh)^2]
\le \frac{K_A^2}{\delta^2}S\gamma
\le\frac{K_A^2}{\delta^2}\one.
\]
Multiply by $w^\top$ and take the supremum over unit $h$. Because $H\succeq0$, this supremum equals $\norm{H}_2$. Alternatively, use $\norm{S}_\infty\le\delta^{-1}$ and $\norm{\gamma}_\infty\le1/4$ separately to obtain $\norm{H}_2\le\norm{w}_1K_A^2/(4\delta^3)$. Taking the smaller estimate proves \eqref{eq:curvature}.
\end{proof}
\begin{remark}
The cancellation avoids multiplying a worst-case resolvent norm by a worst-case activation curvature independently. A softplus unit with slope close to one has small intrinsic curvature. The quadratic estimate is the smaller one precisely when $\delta<1/4$, or $q>3/4$. For $q\le3/4$ the direct cubic estimate may be tighter. The norms $\norm{w}_1$ and $K_A$ must remain in the statement: contraction alone does not uniformly bound the Hessian over arbitrary readouts and input weights.
\end{remark}

\begin{theorem}[Optimality of the exponent]\label{thm:sharpness}
For $0<\eps<1$, take one input and one hidden unit with
\[
W=q=1-\eps,\quad A=w=1,\quad a=c=0,\quad
b_\eps=\log(1-\eps)-\eps\log\eps.
\]
At $x=0$,
\begin{equation}\label{eq:sharpness}
z^*(0)=-\log\eps,\qquad
u_\eps''(0)=\frac{1-\eps}{\eps^2(2-\eps)^3},\qquad
\lim_{\eps\downarrow0}\eps^2u_\eps''(0)=\frac18.
\end{equation}
Thus no bound $\norm{D^2u_\theta}_2\le C\norm{w}_1K_A^2(1-q)^{-p}$ with $p<2$ and a constant $C$ independent of the admissible parameters holds for this class.
\end{theorem}
\begin{proof}
With $z=-\log\eps$, the preactivation is
\[
s=(1-\eps)(-\log\eps)+b_\eps
=\log\frac{1-\eps}{\eps}.
\]
Then $\phi(s)=-\log\eps=z$ and $\phi'(s)=1-\eps$. Uniqueness identifies this state with $z^*(0)$. For a scalar input, \eqref{eq:firstsystem} gives $z'=\phi'(s)/(1-q\phi'(s))$, hence $s'=qz'+1=1/(1-q\phi'(s))$. Equation~\eqref{eq:secondsystem} yields
\[
z''=\frac{\phi'(s)(1-\phi'(s))}{(1-q\phi'(s))^3}.
\]
Substitute $q=\phi'(s)=1-\eps$ and $1-q\phi'(s)=\eps(2-\eps)$ to obtain \eqref{eq:sharpness}. If a bound with exponent $p<2$ held, then $\eps^2u_\eps''(0)\le C\eps^{2-p}\to0$, contradicting the limit $1/8$.
\end{proof}

\begin{remark}[Two-dimensional embedding and meaning of sharpness]
Two independent copies with $W=(1-\eps)I_2$, $A=I_2$ and $w=(1,1)^\top$ give a strictly convex two-dimensional potential with
\[
D^2u_\eps(0)=u_\eps''(0)I_2,\qquad
\det D^2u_\eps(0)\sim\frac{1}{64\eps^4}.
\]
Thus the curvature scaling also occurs within the two-dimensional class. The family varies parameters with $\eps$ and is not a training trajectory. Theorem~\ref{thm:sharpness} proves optimality of the exponent; it does not prove optimality of the constant in Theorem~\ref{thm:curvature}.
\end{remark}

\begin{figure}[!htbp]
\centering\includegraphics[width=\linewidth]{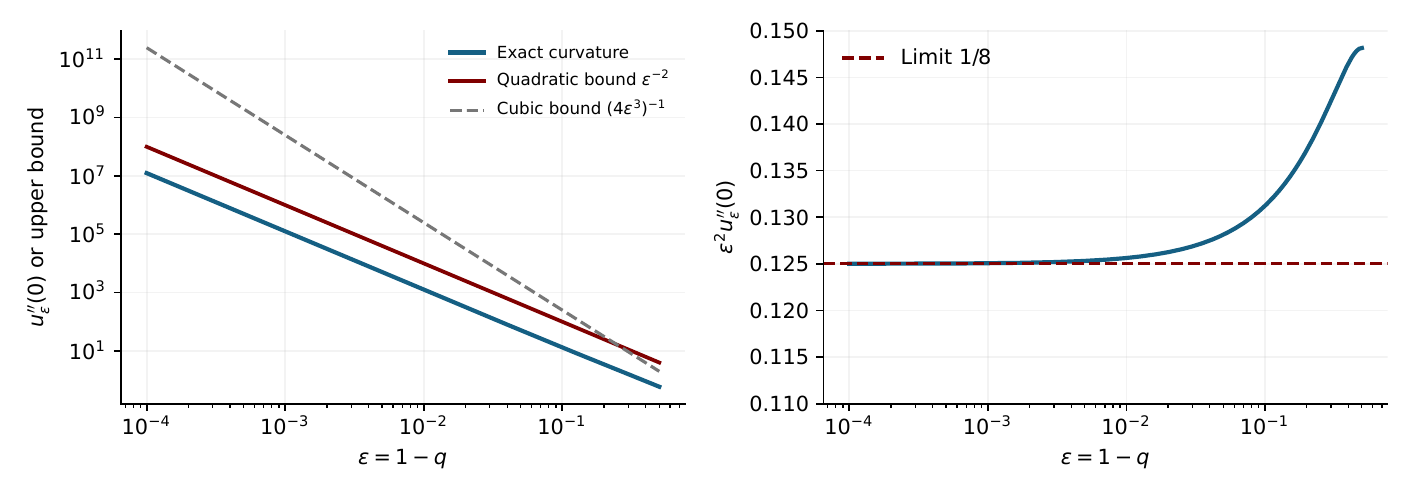}
\caption{Analytic sharpness family from Theorem~\ref{thm:sharpness}. Left: exact curvature and the two upper bounds for $A=w=1$. Right: the rescaled curvature tends to $1/8$ as the contraction gap decreases. Curves are evaluations of the displayed closed forms.}
\label{fig:sharpness}
\end{figure}

\subsection{Sharp third-derivative bound}\label{sec:third}
For a scalar trilinear form we use the norm obtained by taking the supremum over three independent Euclidean unit vectors.
\begin{theorem}[Third derivatives]\label{thm:third}
Under assumption~\eqref{eq:hypotheses}, uniformly in $x$,
\[
\norm{D^3u_\theta(x)}\le (1+3q)\norm{w}_1 K_A^3\delta^{-3}.
\]
The inverse-gap exponent three is optimal over this class.
\end{theorem}
\noindent The proof is given in Appendix~\ref{app:thirdproof}.

\paragraph{Extension to weighted norms.}
The positive cancellation also gives a curvature estimate in a weighted norm whenever a positive vector $\eta$ satisfies $W\eta\le q\eta$. Such a vector exists for every $\rho(W)<q<1$. Appendix~\ref{sec:weighted} states and proves the extension, including the normalization and the spectral-radius endpoint caveat. The numerical experiments use the unweighted row norm.

\section{A posteriori evaluation from an approximate state}\label{sec:certificate}

Fix $x\in\R^2$ and parameters satisfying \eqref{eq:hypotheses}.
We consider numerical approximations to the equilibrium state
and its first and second derivatives with respect to the input:
\[
\hat z\approx z^*(x),\qquad
\hat J_i\approx \partial_{x_i}z^*(x),\qquad
\hat K_{ij}\approx \partial_{x_i}\partial_{x_j}z^*(x),
\qquad i,j\in\{1,2\}.
\]
Here, $\hat J_i$ denotes the $i$th column of an approximate
Jacobian and approximates the exact derivative $J_i=\partial_{x_i}z^*(x)$
defined earlier. We assume $\hat K_{12}=\hat K_{21}$.

For example, $\hat z$ may be obtained by terminating the
equilibrium iteration after finitely many steps, while
$\hat J_i$ and $\hat K_{ij}$ may be obtained by approximately
solving the linear systems arising from implicit differentiation.
The estimates below concern numerical errors in evaluating
a fixed neural representation, rather than the approximation
error between $u_\theta$ and the PDE solution $u^*$.
The approximate quantities need not solve their defining
equations exactly and need not have been produced by the
same algorithm.

Define
\begin{equation}\label{eq:hatdefs}
\begin{gathered}
\hat s=W\hat z+Ax+b,\quad
\hat \Psi=\diag(\phi'(\hat s)),\quad
\hat\gamma=\phi''(\hat s),\quad
\hat B=I-\hat \Psi W,\\
\hat v_i=A_{:,i}+W\hat J_i,\qquad
U_i=\norm{\hat v_i}_\infty.
\end{gathered}
\end{equation}
The residuals, all computed at this same approximate state, are
\begin{align}
\rho_z
&=\norm{\hat z-\phi(\hat s)}_\infty,
\label{eq:rhozero}\\
\rho_i
&=\norm{\hat B\hat J_i-\hat \Psi A_{:,i}}_\infty,
\label{eq:rhofirst}\\
\rho_{ij}
&=\norm{\hat B\hat K_{ij}
-\hat\gamma\odot\hat v_i\odot\hat v_j}_\infty.
\label{eq:rhosecond}
\end{align}
Introduce the following nonnegative quantities:
\begin{align}
E_z&=\rho_z/\delta,
&\alpha&=qE_z/4,
\label{eq:ezalpha}\\
E_i&=(\rho_i+\alpha U_i)/\delta,
&\Delta_i&=qE_i,
\label{eq:ei}\\
Q_{ij}
&=\alpha U_iU_j
+\tfrac14(\Delta_iU_j+U_i\Delta_j+\Delta_i\Delta_j),
\label{eq:qij}\\
E_{ij}
&=\frac{\rho_{ij}+Q_{ij}
+\alpha q\norm{\hat K_{ij}}_\infty}{\delta}.
\label{eq:eij}
\end{align}
All these expressions use the approximate quantities,
parameters and a valid contraction bound $q$.

\begin{theorem}[State-aware derivative certificate]\label{thm:certificate}
Under the preceding assumptions,
\begin{equation}\label{eq:derivativeerrors}
\norm{z^*-\hat z}_\infty\le E_z,\quad
\norm{J_i-\hat J_i}_\infty\le E_i,\quad
\norm{K_{ij}-\hat K_{ij}}_\infty\le E_{ij}.
\end{equation}
For $\hat H_{ij}=w^\top\hat K_{ij}$,
\begin{equation}\label{eq:hessianerror}
\norm{H-\hat H}_2\le E_H,
\qquad E_H=\norm{w}_1\sqrt{E_{11}^2+2E_{12}^2+E_{22}^2}.
\end{equation}
The value approximation $\hat u=w^\top\hat z+a^\top x+c$ satisfies $|u_\theta(x)-\hat u|\le\norm{w}_1E_z$.
\end{theorem}
\noindent The proof is given in Appendix~\ref{app:certproofs}.

\begin{corollary}[Determinant, residual and convexity bounds]\label{cor:det}
In two dimensions set
\begin{equation}\label{eq:edet}
E_{\det}=2\norm{\hat H}_2E_H+E_H^2.
\end{equation}
Then $|\det H-\det\hat H|\le E_{\det}$. For any prescribed $f(x)$, let $\hat r=|\det\hat H-f(x)|$ and $r=|\det H-f(x)|$. One has
\begin{equation}\label{eq:residualinterval}
\max\{0,\hat r-E_{\det}\}\le r\le\hat r+E_{\det},
\end{equation}
and
\begin{equation}\label{eq:eigenbound}
\lambda_{\min}(H)\ge\lambda_{\min}(\hat H)-E_H.
\end{equation}
In particular, a strictly positive right-hand side in \eqref{eq:eigenbound} certifies positive definiteness at that input.
\end{corollary}
\noindent The proof is given in Appendix~\ref{app:certproofs}.

\paragraph{Why state error cannot be omitted.}
Even exact solves of the systems constructed with $\hat B$, $\hat \Psi $ and $\hat\gamma$ only make $\rho_i$ and $\rho_{ij}$ vanish. They do not replace those coefficients by the exact-state coefficients. The $\alpha$ terms in \eqref{eq:ei}--\eqref{eq:eij} account for this discrepancy. If $q=0$, the derivative coefficients are independent of the state and these terms correctly vanish.

\paragraph{Arithmetic scope.}
Theorem~\ref{thm:certificate} and Corollary~\ref{cor:det} are inequalities for exact real quantities. Treating finite-precision residuals as exact upper bounds is insufficient for a validated machine certificate. Such an implementation must enclose residual evaluation, nonlinear functions, parameter norms, linear algebra and determinant/eigenvalue computations with outward rounding or explicit error bounds. The regional implementation described next supplies such enclosures for the saved model. Float64 diagnostic checks remain distinct from these validated bounds.

\section{Local bounds, complete coverage and solution error}\label{sec:regional}
We extend the pointwise certificate to boxes and combine local guarantees over a complete cover. An alternative analytical construction based on local derivative envelopes is developed in Appendix~\ref{sec:local}; it is separate from the interval-residual method evaluated here.

\subsection{A regional residual route with controlled arithmetic}
The regional method uses the following extension of the approximate-state certificate. It does not require third derivatives or the analytical envelopes of Appendix~\ref{sec:local}.
\begin{proposition}[Uniform certificate on a box]\label{prop:box}
Fix a box $X\subset\R^2$ and constant candidates $\hat z,\hat J_i,\hat K_{ij}$, with symmetric $\hat K$. Suppose certified upper bounds for the residual norms in \eqref{eq:rhozero}--\eqref{eq:rhosecond} hold for every $x\in X$, with the coefficients recomputed at $W\hat z+Ax+b$. Substitution of these upper bounds and outward enclosures of all nonnegative constants into \eqref{eq:ezalpha}--\eqref{eq:edet} gives uniform derivative and determinant error bounds on $X$. In particular, if $F(X)$ encloses $f(X)$ and $\widetilde H_{ij}=w^\top\hat K_{ij}$, then
\[
\sup_{x\in X}|\det D^2u_\theta(x)-f(x)|
\le \sup_{y\in F(X)}|\det\widetilde H-y|+E_{\det}(X).
\]

To preserve rigor in finite-precision arithmetic, the determinant
$\det\widetilde H$ must be evaluated as an interval guaranteed
to contain its exact value. The right-hand side must then
be bounded from above using outward-rounded interval arithmetic,
so that rounding errors cannot lead to an underestimate
of the residual bound.
\end{proposition}
\noindent The proof is given in Appendix~\ref{app:regionalproofs}.

Rigorous evaluation uses Arb ball arithmetic with 128-bit midpoint precision \citep{johansson2017}. Effective binary64 parameters are imported as exact dyadic rationals and define the frozen real model being certified. Positivity and $q<1$ are rechecked. This interpretation does not certify the rounding history of optimization. Nonlinear functions, residuals, coefficient norms and determinants are enclosed; scalar exports to binary64 are rounded outward. The implementation uses $\norm{\widetilde H}_F$ as a safe upper bound for its spectral norm. A reported floating determinant receives a further enclosure of its difference from the exact candidate determinant.

The initial state box $Z=\hat z+[-E_z,E_z]^m$ contains every $z^*(x)$ for $x\in X$. Intersections
$Z\leftarrow Z\cap\phi(WZ+AX+b)$ preserve that property. Verified interval linear solves \citep{rump2010} for first and second derivatives may tighten the result; if a solve cannot be verified, the residual bound is retained. This is the implemented local-box method.

\begin{corollary}[Complete disk and boundary cover]\label{cor:cover}
Let $\{X_j\}$ cover the closed unit disk, with residual bounds $R_j$ from Proposition~\ref{prop:box}. Let $\{I_k\}$ cover $[0,2\pi]$, and suppose $B_k$ bounds $|u_\theta(\cos t,\sin t)-g(\cos t,\sin t)|$ for every $t\in I_k$. Then
\[
\norm{\det D^2u_\theta-f}_{L^\infty(B_1)}\le R:=\max_jR_j,
\qquad \norm{u_\theta-g}_{L^\infty(\partial B_1)}\le B_\partial:=\max_k B_k.
\]
\end{corollary}
\noindent The proof is given in Appendix~\ref{app:regionalproofs}.

The domain partition subdivides $[-1,1]^2$ into exact dyadic boxes and discards one only if an exact squared-distance test puts the entire closed box strictly outside the disk. Circle-touching boxes are retained. Boundary arcs use interval $\pi$, sine and cosine. Exact partition audits check coverage. At a work limit, coarse finite bounds are retained, not discarded. Coverage and tolerance attainment are therefore checked independently. A complete cover can yield a valid but uninformative upper bound.

\subsection{From a uniform residual to a solution bound}\label{sec:stability}
The stability estimate of \citet[Proposition 6.1]{nystrom2023}, under its convexity, continuity and data hypotheses, gives
\[
\norm{u_\theta-u^*}_{L^\infty(\Omega)}
\le B_\partial+C_n\operatorname{diam}(\Omega)|\Omega|^{1/n}R^{1/n}.
\]
Here the network is viewed as the exact convex solution with its own right-hand side $\det D^2u_\theta$ and its own boundary trace. Smoothness and nonnegativity of the Hessian ensure that its classical determinant is its Aleksandrov density. The constant $C_n$ depends only on the dimension but is not numerically explicit; on the unit disk a direct barrier gives an explicit coefficient. We give the classical smooth statement, sufficient for P1 and P2; extension to nonsmooth Aleksandrov targets requires the corresponding measure comparison and sum inequality.
\begin{proposition}[Explicit unit-disk barrier]\label{prop:barrier}
Let $u,v\in C^2(B_1)\cap C(\overline B_1)$ be convex and suppose that
\[
\norm{\det D^2u-\det D^2v}_{L^\infty(B_1)}\le R,
\qquad \norm{u-v}_{L^\infty(\partial B_1)}\le B_\partial.
\]
Then
\[
\boxed{\norm{u-v}_{L^\infty(B_1)}\le B_\partial+\tfrac12\sqrt R.}
\]
\end{proposition}
\noindent The proof is given in Appendix~\ref{app:regionalproofs}.

For smooth targets on the unit disk, this standard comparison argument provides an explicit coefficient while retaining square-root dependence on the residual. The numerical refinement evaluates this barrier bound using interval arithmetic and outward rounding, using the previously certified $R$ and $B_\partial$. For P1/P2 we also bound $|u_\theta-u^*|$ directly on boxes using the known manufactured solution; that separate bound does not require a stability constant.

\section{Evaluation procedure and stopping criteria}\label{sec:procedure}
An implementation can factor a single $m\times m$ matrix $\hat B$ at a given input and reuse it for the two first-derivative and three distinct second-derivative right-hand sides. Forming an explicit inverse is unnecessary. For dense matrices, factorization has the usual cubic cost in $m$.

\begin{enumerate}
\item Obtain an approximate state $\hat z$ and calculate \eqref{eq:hatdefs} and $\rho_z$.
\item Obtain $\hat J_i$ and symmetric $\hat K_{ij}$, either by solving the approximate-state systems or by differentiating a finite unroll.
\item Evaluate \eqref{eq:rhofirst}--\eqref{eq:rhosecond} using the same $\hat z$ and form \eqref{eq:ezalpha}--\eqref{eq:eij}.
\item Form $\hat H$, $E_H$ and $E_{\det}$, then report the interval \eqref{eq:residualinterval} rather than only $\hat r$.
\item If the required evaluation tolerance is not met, refine the state and/or derivative computations and recompute the bounds.
\end{enumerate}

An absolute stopping rule is $E_{\det}\le\tau_{\mathrm{abs}}$. If $\hat r>0$, a relative-to-observed-residual criterion is $E_{\det}\le\eta\hat r$ for $0<\eta<1$, which guarantees
\[
(1-\eta)\hat r\le r\le(1+\eta)\hat r.
\]
The corresponding ratio relative to the true residual is bounded by $\eta/(1-\eta)$, not by $\eta$. Near $\hat r=0$, an absolute tolerance remains necessary. No monotonicity of the computed bounds with iteration count is assumed.

For a finite input set $X$, Corollary~\ref{cor:det} gives
\[
\max_{x\in X}r(x)\le\max_{x\in X}\bigl(\hat r(x)+E_{\det}(x)\bigr).
\]
Replacing $X$ by a continuous domain requires additional control between sample points. Section~\ref{sec:regional} supplies that control through regional certificates and a complete cover. Appendix~\ref{sec:local} develops an alternative analytical construction. A sampled analogue for explicit ICNNs, combining Lipschitz constants with covering radii, is given in \citet[Lemma 6.2]{nystrom2023}.

\section{Numerical experiments}\label{sec:numerics}
We investigate three questions: how the a posteriori bounds compare with observed evaluation errors, whether adaptive stopping reduces total evaluation time, and how pointwise guarantees extend to the full domain. A synthetic consistency experiment tests the derivative estimates. Two manufactured Dirichlet problems then assess accuracy and certification on trained networks. Float64 diagnostics and interval guarantees are reported separately.

\subsection{Consistency of the derivative bounds}
We generate 120 random networks in double precision using seed 20261015, with 40 networks at each $q\in\{0.3,0.7,0.95\}$, all with $m=8$ and $n=2$. The entries of a positive matrix are uniform on $[0.05,1]$ and its rows are normalized to sum to $q$. Entries of $A$ and $b$ are independent zero-mean Gaussians with standard deviation $0.5$, readout weights are uniform on $[0.1,1]$, and the input is uniform on $[-1,1]^2$.

A reference state uses fixed-point iteration until its computed infinity-norm residual is at most $10^{-13}$. Reference derivatives solve \eqref{eq:firstsystem}--\eqref{eq:secondsystem}. An approximate state uses $4,8,12$ or $16$ iterations, cycling through these depths. Its derivative solves are perturbed by Gaussian noise with a shared scale $10^U$, where $U$ is uniform on $[-7,-3]$ for each case; second-derivative noise is symmetrized. This probes both state error and nonzero linear-system residuals.

\begin{table}[tb]
\centering\begin{tabular}{lrr}
\toprule Quantity & Maximum ratio & Violations \\
\midrule
State error & 0.98924 & 0 \\
First derivative error & 0.77929 & 0 \\
Second derivative error & 0.69098 & 0 \\
Hessian error & 0.17737 & 0 \\
Determinant error & 0.07730 & 0 \\
Curvature bound & 0.20601 & 0 \\
Cancellation $S\gamma\le\mathbf 1$ & 0.65309 & 0 \\
\bottomrule
\end{tabular}

\caption{Independent float64 consistency experiment, 120 networks. Error rows show the largest observed error-to-bound ratio, maximizing over coordinates where applicable. The last row is $\max_r(S\gamma)_r$, whose bound is one. A violation means a ratio exceeding $1+10^{-8}$. These are numerical checks against approximate reference quantities.}
\label{tab:checks}
\end{table}

Table~\ref{tab:checks} reports no violations. The largest computed reference-state residual bound $\rho_z/(1-q)$ is $1.99\times10^{-12}$. The maximum spectral difference between the linear-system Hessian and its factorization, divided by $\max\{1,\norm{H}_2\}$, is $2.10\times10^{-15}$.

\subsection{Problems, training and evaluation protocol}\label{sec:full}
For both problems, $\Omega=B_1(0)$ and $g=u^*|_{\partial\Omega}$. The manufactured solutions are
\begin{align}
\mathrm{P1}:\quad u^*(x)&=x_1^2+\tfrac14x_2^2,& f(x)&=1,\label{eq:p1}\\
\mathrm{P2}:\quad u^*(x)&=\tfrac12\norm{x}_2^2+0.1(x_1^4+x_2^4),&
f(x)&=(1+1.2x_1^2)(1+1.2x_2^2).\label{eq:p2}
\end{align}
Their Hessians are $\diag(2,1/2)$ and $\diag(1+1.2x_1^2,1+1.2x_2^2)$, respectively. Both targets are smooth and strictly convex, so Proposition~\ref{prop:barrier} applies.

We train one model per problem and seed $0,1,2$, with $m=12$, nominal $q=0.7$, and $L=32$ iterations initialized at zero. The parametrization
\[
W_{rs}=q\,\sigma(\xi_r)\frac{e^{M_{rs}}}{\sum_t e^{M_{rt}}},
\qquad w_r=\frac{\phi(\nu_r)}m
\]
uses unconstrained $\xi,M,\nu$, with $\sigma$ the logistic function, and ensures nonnegative weights with row sums below $q$ in exact arithmetic. The effective saved parameters are rechecked during interval certification. Training minimizes
\begin{equation}\label{eq:loss}
\mathcal L(\theta)=\frac1{N_\Omega}\sum_{x\in X_\Omega}
\bigl(\det D^2u_{\theta,L}(x)-f(x)\bigr)^2
+\frac{10}{N_\partial}\sum_{x\in X_\partial}
\bigl(u_{\theta,L}(x)-g(x)\bigr)^2.
\end{equation}
Interior values of $u^*$ are used only for evaluation. Training differentiates the finite unroll; the study does not benchmark an equilibrium implicit backward pass.

\begin{table}[tb]
\centering\small
\begin{tabularx}{\linewidth}{@{}lX@{}}\toprule
Item & Executed configuration\\\midrule
Adam & 1,200 steps; initial learning rate 0.02; factor 0.35 at steps 780 and 1,080; fresh 128 interior and 64 boundary points per step\\
L-BFGS & Up to 100 iterations, strong-Wolfe line search; fixed 512 interior and 128 boundary points\\
Evaluation & 2,048 interior points and 512 boundary points per model\\
Depth sweep & $L\in\{4,8,12,16,24,32\}$ on the same 2,048 interior points\\
Float64 timing & 512 points/model; five repetitions; tolerances $10^{-4}$ and $10^{-6}$\\
Interval timing & Separate eight-point/model cohort; five repetitions; same tolerances\\
Arithmetic & CPU, float64, one computational thread; Arb at 128-bit precision\\\bottomrule
\end{tabularx}
\caption{Protocol of the completed six-model run.}
\label{tab:training}
\end{table}

Interior evaluation points are sampled uniformly in the disk using the reserved seed $900000+\mathrm{seed}$. Boundary angles are $2\pi(k+0.371)/512$, $k=0,\ldots,511$. Training and evaluation samples are distinct. Reference evaluations use converged float64 states and linear solves; the numerical stopping test for the state correction divided by $1-q$ is $10^{-13}$. These reference values are diagnostics, not exact real values. The maximum reference determinant diagnostic across the models is $1.65\times10^{-11}$ (rounded up).

\subsection{Accuracy of the six trained models}
Table~\ref{tab:pilots} reports sample metrics for the saved equilibrium models. The relative error is $\norm{u-u^*}_{\ell^2(X)}/\norm{u^*}_{\ell^2(X)}$. Its mean and sample standard deviation across the three seeds are $0.2847\%\pm0.0333\%$ for P1 and $0.2941\%\pm0.0178\%$ for P2. These are descriptive summaries over three runs, not uniform error guarantees.

\begin{table}[tb]
\centering\small\setlength{\tabcolsep}{5pt}
\begin{tabular}{@{}lrrrrr@{}}\toprule
Model & Rel. $\ell^2$ (\%) & Max. $|u-u^*|$ & Max. $|\det H-f|$ & Max. boundary & $\norm W_\infty$ \\ \midrule
P1/0 & 0.2478 & 0.003927 & 0.045998 & 0.004150 & 0.6584 \\
P1/1 & 0.2937 & 0.004282 & 0.054399 & 0.004573 & 0.6135 \\
P1/2 & 0.3125 & 0.004958 & 0.052480 & 0.005022 & 0.5817 \\
P2/0 & 0.2808 & 0.003626 & 0.102253 & 0.003961 & 0.6466 \\
P2/1 & 0.2872 & 0.003799 & 0.080306 & 0.003872 & 0.6580 \\
P2/2 & 0.3144 & 0.005618 & 0.089144 & 0.005954 & 0.6330 \\
\bottomrule\end{tabular}

\caption{Float64 sample metrics. A model label P1/0 means problem P1, seed 0. Interior maxima use 2,048 points; boundary maxima use 512 points. Values are rounded for display and are not interval bounds.}
\label{tab:pilots}
\end{table}

The smallest sampled Hessian eigenvalues range from $0.4566$ to $0.4756$ for P1 and from $0.9801$ to $1.0003$ for P2. Global convexity follows from the representation theorem, not from these samples. The learned row norms range from $0.5817$ to $0.6584$, all below $3/4$. Consequently, the cubic branch in \eqref{eq:curvature} is smaller for these trained models: the experiments test the evaluation certificate, while the near-critical exponent sharpness is demonstrated by the analytical family in Theorem~\ref{thm:sharpness}.

\subsection{Pointwise certificates and iteration depth}\label{sec:evidence}
There are $6\times6\times2048=73{,}728$ point/depth pairs. Every pair has both a float64 diagnostic and an Arb certificate. The float64 diagnostic declares a violation only if the observed determinant error exceeds its estimator plus the reference diagnostic and a $10^{-11}$ numerical floor. No such violations are observed. The interval audit contains 73,728 finite certificate records at 128-bit precision, with no missing point/depth pairs.

\begin{figure}[!htbp]
\centering\includegraphics[width=\linewidth]{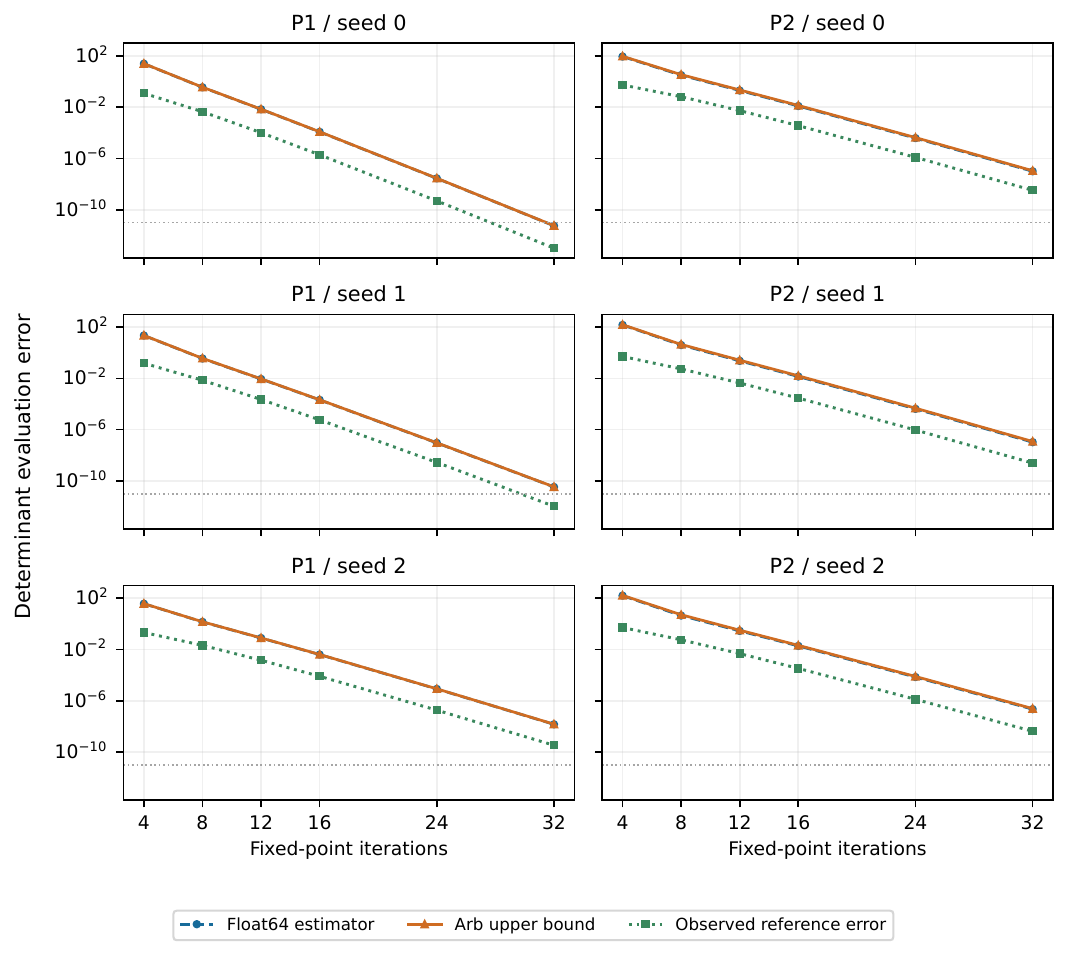}
\caption{Determinant evaluation error over iteration depth for all six models. Each curve is the maximum over the same 2,048 evaluation points. Blue: float64 estimator; orange: outward-rounded Arb upper bound; green: error against the float64 reference. The gray line marks $10^{-11}$, the additive diagnostic floor; the reference diagnostic is added separately in violation counting. Curves near this scale should not be interpreted as exact-error measurements.}
\label{fig:depth}
\end{figure}

Figure~\ref{fig:depth} shows decreasing evaluation errors with depth. At depth 12, the median float64 bound/error ratio over points above the comparison floor ranges from $26.15$ to $58.03$ across the six models. Ratios at the floating-point floor are excluded from this statistic.

A fixed iteration threshold does not certify every point relative to its PDE residual. At depth 12, $62.99$--$89.11\%$ of P1 points satisfy $E_{\det}\le0.05\hat r$, depending on the seed. At depth 16, the corresponding P2 fractions are $86.52$--$92.38\%$. All six sampled cohorts satisfy this float64 criterion at depth 32. Absolute tolerances remain necessary near vanishing residuals, as explained in Section~\ref{sec:procedure}.

\FloatBarrier
\subsection{Cost of adaptive evaluation}
The float64 benchmark compares fixed 32 iterations, fixed 32 plus the diagnostic certificate, and adaptive stopping on checkpoints $4,8,12,16,24,32,48,64$. The timer includes the certificate and management of the remaining active points. Methods are warmed up, and their timing order is permuted deterministically across five repetitions. We use the ratio of median times, $t_{32}/t_{\mathrm{adaptive}}$, as the speedup measure. The baseline for this cohort is fixed-32 evaluation without a certificate.

At tolerance $10^{-4}$, adaptive float64 evaluation uses mean depths $11.66$--$16.05$ and is $1.45$--$1.99$ times faster across models. At $10^{-6}$ it uses mean depths $14.95$--$20.84$ and is $1.13$--$1.48$ times faster. All 512 points per model are accepted at both tolerances, and observed errors against the reference lie below the requested tolerances. This stopping rule uses the float64 estimator without roundoff enclosure.

\begin{figure}[!htbp]
\centering\includegraphics[width=\linewidth]{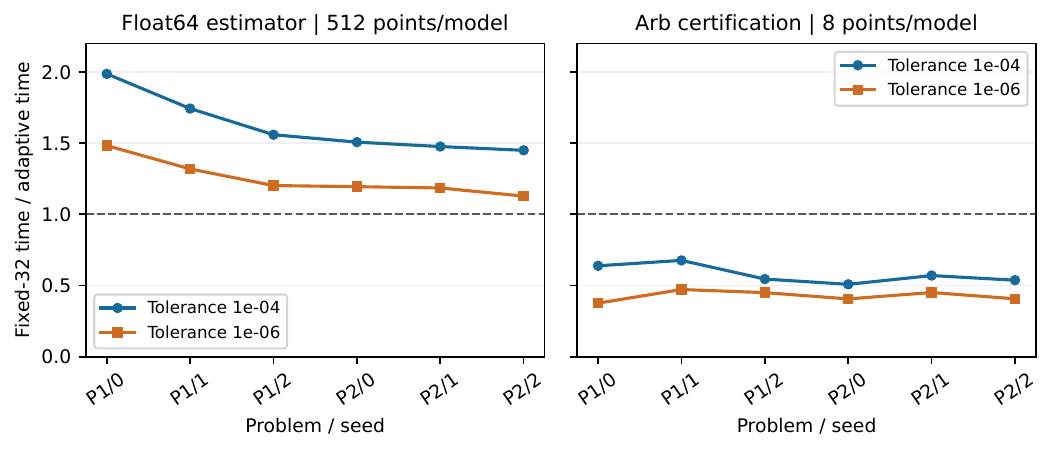}
\caption{Matched-cohort evaluation timings. Ratios above one favor adaptive evaluation; ratios below one favor fixed 32 iterations. Left: float64 estimator on 512 points/model, compared with fixed 32 without certification. Right: full Arb stopping on eight points/model, compared with fixed 32 plus Arb certification. Both comparisons include the adaptive certificate cost and use five repetitions. The cohorts and baselines differ; their throughputs must not be compared directly.}
\label{fig:timing}
\end{figure}

The separate Arb benchmark includes all interval work in both fixed and adaptive timings. All eight points per model are certified within each requested tolerance. Nevertheless, speedup ratios are $0.509$--$0.677$ at $10^{-4}$ and $0.376$--$0.472$ at $10^{-6}$: adaptive Arb evaluation is approximately $1.48$--$2.66$ times slower. Repeated certification outweighs the iteration savings in this implementation and cohort. The interval timing loops use fixed-then-adaptive order after warmup, unlike the interleaved float64 benchmark; timing-order effects and the small cohort limit broader performance conclusions. 

\FloatBarrier
\subsection{Complete-domain and boundary certification}\label{sec:domainresults}
The initial domain protocol uses exact dyadic boxes, 128-bit Arb calculations, depth at most 8, a 20,000-box work budget, and residual target $R\le0.25$. Boundary refinement uses trigonometric interval arcs, depth at most 12, an 8,192-arc budget and target $B_\partial\le0.01$. All six initial covers are complete, all satisfy the boundary target, and five satisfy the residual target. Initial domain-stage times range from 62.7 to 202.1 seconds in the original runtime.

For P2, seed 2, the initial residual bound is $0.4376936062$, with 14,584 retained boxes and 502 boxes above the requested tolerance. We therefore perform a separately recorded, post hoc refinement of that fixed model. Only boxes above $0.25$ are subdivided, prioritizing the largest bounds. The first stage permits depth 8; the second permits depth 9 if needed. Each stage is limited to 8,000 child-box evaluations and 120 seconds, with a cooperative time check before each complete four-child split. A parent is replaced only after all its intersecting children have certificates. Its valid bounds also apply to each child and are retained when tighter. Boundary certificates and model parameters are unchanged.

\begin{figure}[!htbp]
\centering\includegraphics[width=.95\linewidth]{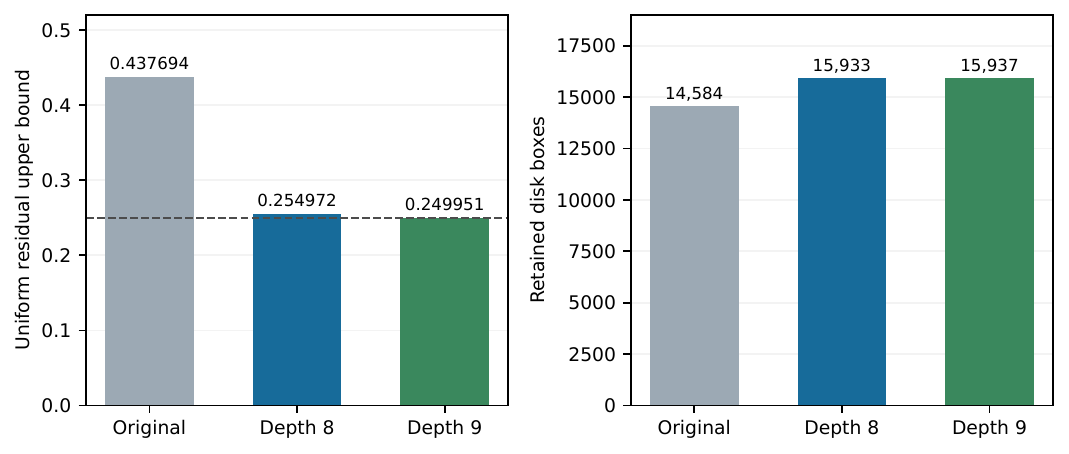}
\caption{Additional P2, seed-2 certification, without retraining. The residual upper bound decreases from $0.437694$ to $0.254972$ at depth 8 and then to $0.249951$ at depth 9 (displayed bounds rounded up). The cover remains complete throughout. Stage 1 computes 1,850 new certificates and stage 2 computes 10, retaining 15,937 boxes at completion.}
\label{fig:refinement}
\end{figure}

The additional refinement uses a separate CPU environment with one computational thread. Its measured incremental computation time is 3.32 seconds, excluding input/output and coverage audits. This reuses the existing domain and boundary certificates, runs in a different runtime, and is not a matched speed comparison with the original full-domain calculation. The exported partition is checked using exact arithmetic.

\begin{table}[tb]
\centering\small\setlength{\tabcolsep}{4pt}
\begin{tabular}{@{}lrrrrr@{}}\toprule
Model & $R$ & $B_\partial$ & Direct $U$ & Barrier $U$ & Boxes \\ \midrule
P1/0 & 0.232852 & 0.010000 & 0.107237 & 0.251273 & 3,332 \\
P1/1 & 0.243891 & 0.009989 & 0.114571 & 0.256915 & 3,332 \\
P1/2 & 0.217254 & 0.009996 & 0.110518 & 0.243048 & 3,332 \\
P2/0 & 0.249993 & 0.009997 & 0.053797 & 0.259994 & 11,310 \\
P2/1 & 0.249997 & 0.009999 & 0.065538 & 0.259997 & 11,668 \\
P2/2 & 0.249951 & 0.009989 & 0.046754 & 0.259965 & 15,937 \\
\bottomrule\end{tabular}

\caption{Uniform upper bounds on the closed disk and its boundary. P2/2 uses the additional refinement; all other rows use the original covers. Direct $U$ bounds $|u_\theta-u^*|$ using the known manufactured solution. Barrier $U=B_\partial+\sqrt R/2$ uses Proposition~\ref{prop:barrier}. Every displayed bound is rounded upward to six decimal places.}
\label{tab:domain}
\end{table}

After refinement, every model satisfies $R<0.25$ and $B_\partial<0.01$ (Table~\ref{tab:domain}). Uniform residual bounds are approximately $2.44$--$5.06$ times the corresponding sampled residual maxima.

\begin{figure}[!htbp]
\centering\includegraphics[width=\linewidth]{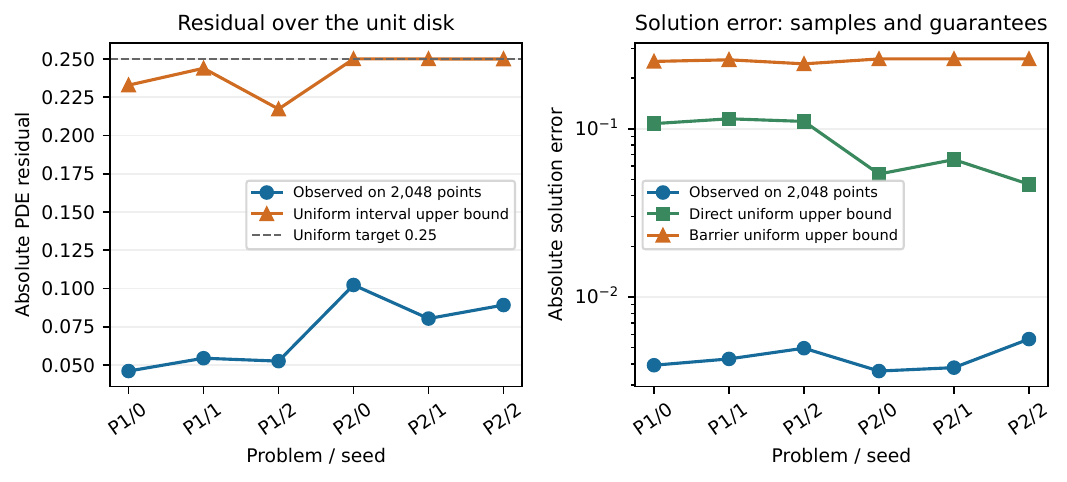}
\caption{Samples and uniform bounds have different meanings. Left: observed PDE residual maxima versus interval upper bounds over the disk. Right: observed solution-error maxima, direct uniform bounds using the manufactured target, and the explicit barrier bound. P2/2 includes the final refinement. Curves connect model labels for readability; no trend across seeds is assumed.}
\label{fig:uniform}
\end{figure}

The barrier expression is evaluated in Arb from the certified binary64 upper bounds and exported upward. For the six final covers it gives solution-error bounds of approximately $0.243$--$0.260$. The direct manufactured-solution bounds, approximately $0.047$--$0.115$, are tighter in every case. In particular, P2/2 has barrier bound $0.2599640704$ and direct bound $0.0467530969$, both rounded up. The direct bound is obtained independently and is unchanged by this refinement.

\FloatBarrier
\FloatBarrier
\section{Discussion and limitations}\label{sec:limitations}
\paragraph{Evaluation and training.}
The results certify evaluation of frozen networks. They do not prove convergence of sampled strong-loss training or control approximation, optimization and loss-discretization errors along a training sequence. Two manufactured problems and three seeds each provide limited empirical coverage, and no comparison with explicit ICNNs or other PDE solvers is conducted here.

\paragraph{Scope and tightness.}
The inverse-gap exponents are sharp over the network class, but optimal prefactors remain open. The weighted result and analytical local envelope are theoretical extensions; the implemented regional path uses unweighted interval residuals and verified solves. The disk barrier is stated for smooth convex targets. Its square-root dependence can give loose solution estimates, and the direct interval bounds also substantially exceed the observed sample errors.

\paragraph{Performance and numerical guarantees.}
Adaptive float64 evaluation reduces measured runtime, while full Arb stopping is slower in the tested implementation. Interval batching, reuse and stopping-policy design are possible future improvements, not demonstrated gains. Outward arithmetic, exact coverage checks and consistency tests support the implementation, but do not constitute formal verification of the complete software stack.

\section{Conclusion}
For nonnegative contractive softplus equilibrium networks, positivity yields $S\gamma\le\one$ and a uniform Hessian bound with optimal inverse-gap exponent two. A residual-based perturbation argument controls derivatives computed at approximate states. Third derivatives admit sharp exponent three, and positive weighted norms extend the representation condition to $\rho(W)<1$. Regional interval residuals and complete boundary covers provide uniform bounds for frozen models, with an explicit residual-to-solution estimate on the disk. The six-model study supplies 73,728 point/depth interval certificates and, after one bounded fixed-model refinement, complete covers meeting the prescribed residual and boundary targets for every model. Adaptive evaluation saves time with the float64 estimator, whereas rigorous adaptive stopping remains more expensive than fixed-depth certification in the measured cohorts. These findings quantify both the usefulness and the cost of certified evaluation.

\section*{Code availability}\label{sec:availability}
The code used to train and evaluate the models is publicly available at \url{https://github.com/EdisonVazquezG/monge-ampere-equilibrium}.

\section*{Use of AI tools}
The authors used ChatGPT (OpenAI) and Claude (Anthropic) to assist with language editing, consistency checks of notation and statements, verification of bibliographic references against their sources. All mathematical results, proofs, experiments and conclusions are the authors' own. The authors reviewed and edited all AI-assisted content and take full responsibility for the manuscript.

\clearpage
\begingroup
\small
\setlength{\bibsep}{3pt}
\bibliographystyle{icml2017}
\bibliography{references}
\endgroup

\clearpage
\appendix
\numberwithin{equation}{section}
\section{Equilibrium foundations and derivative identities}\label{app:foundations}
\subsection{Softplus and the equilibrium}
\begin{lemma}[Derivatives of the Softplus function $\phi$]\label{lem:softplus}
\[
0<\phi'(t)<1,\qquad \phi''(t)=\phi'(t)(1-\phi'(t))\in(0,1/4],\qquad
\phi'''(t)=\phi'(t)(1-\phi'(t))(1-2\phi'(t)).
\]
In particular $\phi$ is increasing, convex and $1$-Lipschitz; both $\phi'$ and $\phi''$ are $1/4$-Lipschitz.
\end{lemma}
\begin{proof}
Differentiate $\phi(t)=\log(1+e^t)$ to obtain $\phi'(t)=e^t/(1+e^t)$ and the displayed identities. The bounds follow from the fact that $a(1-a)\le 1/4$ and $|1-2a|\le 1$ for any $a\in (0,1)$. The mean value theorem gives the Lipschitz statements.
\end{proof}

\begin{proposition}[Existence, convexity and state error]\label{prop:equilibrium}
Under assumption~\eqref{eq:hypotheses}, for every $x\in\R^n$ the fixed-point equation in \eqref{eq:model} has a unique solution. Each component of $z^*(x)$ and the scalar function $u_\theta(x)$ is convex. For any approximate state $\hat z\in\R^m$, define
\[
\rho_z=\norm{\hat z-\phi(W\hat z+Ax+b)}_\infty.
\]
Then
\begin{equation}\label{eq:statebound}
\norm{z^*(x)-\hat z}_\infty\le\frac{\rho_z}{\delta}.
\end{equation}
Moreover, $\Lip(u_\theta)\le\norm{w}_1K_A/\delta+\norm{a}_2$, where the input space carries the Euclidean norm.
\end{proposition}
\begin{proof}
For fixed $x$, let $T_x(z)=\phi(Wz+Ax+b)$. Lemma~\ref{lem:softplus} gives
\[
\norm{T_x(z)-T_x(\tilde z)}_\infty
\le\norm{W}_\infty\norm{z-\tilde z}_\infty
\le q\norm{z-\tilde z}_\infty.
\]
The space $\R^m$ is complete, so Banach's theorem gives a unique fixed point and convergence of the iterates. Starting from the constant function $z_0(x)=0$, each component of $Wz_\ell(x)+Ax+b$ is convex whenever those of $z_\ell$ are convex. Composition with the increasing convex $\phi$ preserves convexity. Pointwise convergence $z_\ell\to z^*$ then gives entrywise convexity of the limit. Nonnegative linear combination and addition of an affine function yield convexity of $u_\theta$.

The triangle inequality and contraction imply
\[
\norm{\hat z-z^*}_\infty\le\rho_z+q\norm{\hat z-z^*}_\infty,
\]
which proves \eqref{eq:statebound}. Finally, comparing the fixed-point equations at $x,y$ gives
\[
\norm{z^*(x)-z^*(y)}_\infty
\le q\norm{z^*(x)-z^*(y)}_\infty+K_A\norm{x-y}_2.
\]
Divide by $\delta$ and apply the readout.
\end{proof}

\subsection{Implicit derivatives and Hessian structure}
\begin{proposition}[Smoothness and derivative systems]\label{prop:derivatives}
The function $z^*$ is $C^\infty$ on $\R^n$. For $1\le i,j\le n$,
\begin{align}
    BJ_i&=\Psi A_{:,i},\label{eq:firstsystem}\\
BK_{ij}&=\gamma\odot v_i\odot v_j,
\qquad v_i=A_{:,i}+WJ_i=V_{:,i}.\label{eq:secondsystem}
\end{align}
Furthermore $S\ge0$, $S\ge I$, $\norm{S}_\infty\le\delta^{-1}$, and $J=\Psi V$.
\end{proposition}
\begin{proof}
Since $0<\Psi<I$ on the diagonal, $\norm{\Psi W}_\infty\le q$ and $\norm{W\Psi}_\infty\le q$. Both inverses in \eqref{eq:notations} therefore exist since the Neumann series converge. In particular
\[
S=\sum_{\ell=0}^\infty(\Psi W)^\ell\ge I,
\qquad \norm{S}_\infty\le\sum_{\ell=0}^\infty q^\ell=\delta^{-1}.
\]
For $F(z,x)=z-\phi(Wz+Ax+b)$ the derivative in $z$ equals $B$ at a fixed point. The implicit function theorem yields a smooth local solution at every input. Uniqueness makes these local solutions agree on overlaps, giving a global smooth $z^*$.

Differentiate $z^*=\phi(s)$ once to obtain $J_i=\Psi (WJ_i+A_{:,i})$, hence \eqref{eq:firstsystem}. A second differentiation gives
\[
K_{ij}=\Psi WK_{ij}+\gamma\odot(WJ_i+A_{:,i})\odot(WJ_j+A_{:,j}),
\]
which proves \eqref{eq:secondsystem}. Finally, the identity $(I-\Psi W)\Psi=\Psi(I-W\Psi)$ implies $S\Psi A=\Psi(I-W\Psi)^{-1}A=\Psi V$. Since $J=S\Psi A$, it also follows that $A+WJ=V$.
\end{proof}

\begin{proposition}[Hessian factorization]\label{prop:factor}
Let $\beta=(S^\top w)\odot\gamma$. Then
\begin{equation}\label{eq:factor}
H:=D^2u_\theta=V^\top\diag(\beta)V\succeq0.
\end{equation}
If all entries of $w$ are strictly positive, then $H\succ0$ if and only if $\rank A=n$. For $n=2$, writing $v_{r1},v_{r2}$ for the entries of row $r$ of $V$,
\begin{equation}\label{eq:cauchybinet}
\det H=\sum_{1\le r<t\le m}\beta_r\beta_t
(v_{r1}v_{t2}-v_{r2}v_{t1})^2.
\end{equation}
\end{proposition}
\begin{proof}
For any $h\in\R^n$, bilinearity of \eqref{eq:secondsystem} gives
\[
D^2z^*[h,h]=S\bigl[\gamma\odot(Vh)^2\bigr].
\]
The affine readout has zero Hessian, so
\[
h^\top Hh=(S^\top w)^\top[\gamma\odot(Vh)^2]
=\sum_r\beta_r(Vh)_r^2.
\]
This proves \eqref{eq:factor}. All $\beta_r$ are nonnegative. If $w>0$, then $S\ge I$ and $\gamma>0$ give $\beta>0$. Consequently $H$ is positive definite exactly when $V$ has full column rank. The invertibility of $I-W\Psi$ gives $\rank V=\rank A$. Formula~\eqref{eq:cauchybinet} is Cauchy--Binet applied to $\diag(\sqrt\beta)V$.
\end{proof}

\begin{remark}
Strict positivity of every entry of $w$ is a sufficient hypothesis for the rank criterion, not a necessary condition for strict convexity of every network. The factorization itself remains valid when some readout weights vanish. The two-dimensional determinant formula is also an algebraic identity, not a roundoff guarantee for its numerical evaluation.
\end{remark}

\section{Proof of the third-derivative bound}\label{app:thirdproof}
This appendix proves Theorem~\ref{thm:third}, including optimality of the inverse-gap exponent.
\begin{proof}
Put $M=K_A/\delta$. For unit vectors $h_1,h_2,h_3$, write
$v_i=Vh_i$, $K_{ij}=D^2z^*[h_i,h_j]$, and
$t=D^3z^*[h_1,h_2,h_3]$. Positivity of $S$ and the cancellation
$S\gamma\le\one$ give
\[
|v_i|\le M\one,\qquad
|K_{ij}|=|S(\gamma\odot v_i\odot v_j)|\le M^2\one.
\]
Differentiating the second-derivative equation gives
\begin{align*}
Bt={}&\phi'''(s)\odot v_1\odot v_2\odot v_3\\
&+\gamma\odot\bigl((WK_{12})\odot v_3+(WK_{13})\odot v_2
+(WK_{23})\odot v_1\bigr).
\end{align*}
Since $|\phi'''(s)|\le\gamma$ and $|WK_{ij}|\le qM^2\one$, multiplication by $S\ge0$ yields
$|t|\le(1+3q)M^3S\gamma\le(1+3q)M^3\one$.
Apply $w^\top$ and take the supremum over all three directions. 

For sharpness use the one-neuron family of Theorem~\ref{thm:sharpness}, for which $q=\phi'(s)=1-\eps$ at $x=0$. Direct scalar differentiation gives
\[
u'''(0)=\frac{\phi'''(s)}{(1-q\phi'(s))^4}
+\frac{3q\,\phi''(s)^2}{(1-q\phi'(s))^5}
=\frac{(1-\eps)(1-\eps+\eps^2)}{\eps^3(2-\eps)^5}.
\]
Thus $\eps^3u'''(0)\to1/32$. This proves optimality of the exponent.
\end{proof}
\section{Positive weighted norms}\label{sec:weighted}
Weighted contraction is a standard tool for implicit networks; see, for example, \citet{jafarpour2021}, and \citet{elghaoui2021} for Perron--Frobenius well-posedness conditions. The result here transports the softplus curvature cancellation to that setting.
\begin{proposition}[Weighted curvature]\label{prop:weighted}
Let $W,w\ge0$ and suppose that there exist $\eta>0$ and $q<1$ with
$W\eta\le q\eta$. Normalize $\max_r\eta_r=1$ and set
\[
\norm{v}_\eta=\max_r |v_r|/\eta_r,\qquad
K_\eta=\max_r\norm{A_{r,:}}_2/\eta_r.
\]
The equilibrium exists uniquely, is smooth and convex, and
\[
\norm{D^2u_\theta}_2\le (w^\top\eta)K_\eta^2(1-q)^{-2}.
\]
Such a positive weight exists for every $q$ with $\rho(W)<q<1$.
\end{proposition}
\begin{proof}
The $\eta$-weighted induced infinity norms of $W$, $\Psi W$ and $W\Psi$ are at most $q$. The earlier contraction, positivity and smoothness arguments therefore apply. Moreover,
\[
\gamma\odot\eta\le(I-\Psi)\eta\le (I-\Psi W)\eta,
\qquad S(\gamma\odot\eta)\le\eta.
\]
For a unit input direction, $|Vh|\le K_\eta\eta/(1-q)$. Because $0<\eta_r\le1$, its entrywise square is at most $K_\eta^2\eta/(1-q)^2$. The directional Hessian formula then gives the stated bound. Finally, for $q>\rho(W)$ the Neumann series gives
$\eta_0=(I-W/q)^{-1}\one>0$, and $W\eta_0=q(\eta_0-\one)\le q\eta_0$; normalization preserves this inequality.
\end{proof}
\begin{remark}[The spectral-radius endpoint need not be attained]
In general, $\inf_{\eta>0}\max_r(W\eta)_r/\eta_r=\rho(W)$, but a strictly positive minimizing vector need not exist. For
$W=\left(\begin{smallmatrix}r&1\\0&r\end{smallmatrix}\right)$ with $0\le r<1$, every positive $\eta$ gives a first-row ratio $r+\eta_2/\eta_1>r$. A positive Perron vector, when available, permits equality.
\end{remark}

\section{Analytical local envelopes}\label{sec:local}
\begin{proposition}[Local envelopes]\label{prop:local}
Let $Q=x_0+[-\rho,\rho]^n$ and let $a_r=\norm{A_{r,:}}_2$.
Write $\Psi(x)=\diag(\phi'(s(x)))$, with $s(x)=Wz^*(x)+Ax+b$. Suppose a diagonal matrix $\overline\Psi$ bounds $\Psi(x)$ entrywise on $Q$ and satisfies $0\le\overline\Psi\le I$. Set
\[
P=(I-W\overline\Psi)^{-1},\qquad e=\sqrt n\rho Pa.
\]
If $s_0=s(x_0)$, then $|s(x)-s_0|\le e$ on $Q$.
Consequently $\diag(\phi'(s_0+e))$ is another valid upper envelope.
Starting at $\overline\Psi=I$, intersection with this new bound preserves validity at every step.
\end{proposition}
\begin{proof}
The Neumann expansion and $\Psi(x)\le\overline\Psi$ imply
$(I-W\Psi(x))^{-1}\le P$. For any unit direction $h$,
$|Ds(x)h|=|(I-W\Psi(x))^{-1}Ah|\le Pa$.
Integrate this inequality along the segment from $x_0$ to $x$, which lies in $Q$, and use $\norm{x-x_0}_2\le\sqrt n\rho$. Since $\phi'$ is increasing, the updated diagonal bounds $\Psi(x)$. The initial envelope is valid since $0<\phi'<1$, and induction proves the assertion. If only an enclosure $s_0\in[\underline s_0,\overline s_0]$ is available, replace the upper endpoint by $\overline s_0+e$ and the lower by $\underline s_0-e$; an unverified center approximation cannot be substituted for $s_0$.
\end{proof}
For any valid final envelope, recompute $P$ and define
$v=Pa$, $\overline S=(I-\overline\Psi W)^{-1}$. Let $\overline\gamma$ bound $\phi''(t)=\phi'(t)(1-\phi'(t))$ on the resulting preactivation intervals, including its maximum $1/4$ when the interval contains zero. Then
\begin{align*}
k&=\overline S(\overline\gamma\odot v^2),
& C_2(Q)&=w^\top k,\\
t&=\overline S\bigl(\overline\gamma\odot(v^3+3v\odot Wk)\bigr),
& C_3(Q)&=w^\top t
\end{align*}
bound the Hessian and third-derivative norms on $Q$. These follow by the same mixed-direction absolute-value arguments used above. In two dimensions,
$\Lip_Q(\det D^2u_\theta-f)\le2C_2(Q)C_3(Q)+\Lip_Q(f)$, whenever $f$ is Lipschitz on $Q$: the derivative of the determinant pairs $\cof(H)$ with a directional derivative of $H$, and $\norm{\cof(H)}_*\le2\norm H_2$.
This gives a center-plus-variation covering construction. A sampled analogue for explicit softplus ICNNs, combining Lipschitz constants of the residual with covering radii, appears in \citet[Lemmas 6.2--6.3]{nystrom2023}; here the derivative envelopes account for the equilibrium resolvent. Its sharpness and computational cost are separate questions.

\section{Proofs of the approximate-state certificates}\label{app:certproofs}
\subsection{Proof of Theorem~\ref{thm:certificate}}
\begin{proof}
The state estimate is Proposition~\ref{prop:equilibrium}. It implies
\[
\norm{s-\hat s}_\infty\le qE_z.
\]
By Lemma~\ref{lem:softplus},
\begin{equation}\label{eq:coefferrors}
\norm{\Psi-\hat \Psi}_\infty\le\alpha,\qquad
\norm{\gamma-\hat\gamma}_\infty\le\alpha.
\end{equation}
Let $r_i=\hat B\hat J_i-\hat \Psi A_{:,i}$. Since $BJ_i=\Psi A_{:,i}$ and $B=\hat B-(\Psi-\hat \Psi)W$, direct subtraction gives
\begin{equation}\label{eq:firstidentity}
B(J_i-\hat J_i)=(\Psi-\hat \Psi)\hat v_i-r_i.
\end{equation}
Multiplication by $S$ and use of $\norm{S}_\infty\le\delta^{-1}$ prove the first-derivative bound. In particular, if $e_i=v_i-\hat v_i=W(J_i-\hat J_i)$, then $\norm{e_i}_\infty\le\Delta_i$.

For the second derivatives, set
\[
r_{ij}=\hat B\hat K_{ij}-\hat\gamma\odot\hat v_i\odot\hat v_j.
\]
Subtracting the approximate equation from \eqref{eq:secondsystem} gives the exact identity
\begin{align}\label{eq:secondidentity}
B(K_{ij}-\hat K_{ij})={}&\gamma\odot v_i\odot v_j
-\hat\gamma\odot\hat v_i\odot\hat v_j\nonumber\\
&+(\Psi-\hat \Psi)W\hat K_{ij}-r_{ij}.
\end{align}
The difference of the nonlinear right-hand sides expands as
\begin{align*}
&\gamma\odot v_i\odot v_j-\hat\gamma\odot\hat v_i\odot\hat v_j\\
&\quad=(\gamma-\hat\gamma)\odot\hat v_i\odot\hat v_j
+\gamma\odot(e_i\odot\hat v_j+\hat v_i\odot e_j+e_i\odot e_j).
\end{align*}
Its infinity norm is at most $Q_{ij}$ by \eqref{eq:coefferrors}, $\norm{\gamma}_\infty\le1/4$ and the bounds on $e_i,e_j$. The remaining coefficient term is at most $\alpha q\norm{\hat K_{ij}}_\infty$. Applying $S$ to \eqref{eq:secondidentity} gives \eqref{eq:eij} and the second-derivative estimate.

Finally, $|H_{ij}-\hat H_{ij}|\le\norm{w}_1E_{ij}$. The symmetry hypothesis implies $E_{12}=E_{21}$, and the Frobenius norm dominates the spectral norm, proving \eqref{eq:hessianerror}. The value estimate follows directly from the readout and the state bound.
\end{proof}
\subsection{Proof of Corollary~\ref{cor:det}}
\begin{proof}
Write $E=H-\hat H$. The two-dimensional determinant identity gives
\[
\det(\hat H+E)-\det\hat H=\cof(\hat H):E+\det E.
\]
For a $2\times2$ matrix, the cofactor matrix has the same two singular values as the original matrix, in reversed order. Its nuclear norm is therefore at most $2\norm{\hat H}_2$. The matrix dual-norm inequality and $|\det E|\le\norm{E}_2^2$ imply \eqref{eq:edet}. The triangle and reverse triangle inequalities give \eqref{eq:residualinterval}. Since both Hessians are symmetric, the Rayleigh-quotient bound $v^\top Ev\ge-\norm{E}_2$ for unit $v$ proves \eqref{eq:eigenbound}.
\end{proof}
\section{Proofs of the regional and solution-error bounds}\label{app:regionalproofs}
\subsection{Proof of Proposition~\ref{prop:box}}
\begin{proof}
Apply Theorem~\ref{thm:certificate} separately at every $x\in X$ with the same constant candidates. Every expression for an error magnitude is nondecreasing in its nonnegative upper-bound inputs, with a strictly positive lower bound for $1-q$. Hence uniform residual upper bounds give uniform output bounds. The determinant claim is the triangle inequality. Interval evaluation provides the required simultaneous enclosures; it is not an assumption that a floating-point residual is already an upper bound.
\end{proof}
\subsection{Proof of Corollary~\ref{cor:cover}}
\begin{proof}
Every point belongs to a certified covering element, so the appropriate local bound is at most the maximum. Parameterization of the circle transfers the angular bounds to the whole boundary.
\end{proof}
\subsection{Proof of Proposition~\ref{prop:barrier}}
\begin{proof}
Set $b(x)=\sqrt R(\norm{x}_2^2-1)/2$. It vanishes on the circle, is convex, and satisfies $-\sqrt R/2\le b\le0$. For every positive semidefinite $2\times2$ matrix $H$,
\[
\det(H+\sqrt R I)=\det H+\sqrt R\operatorname{tr}H+R\ge\det H+R.
\]
Thus $u+b-B_\partial$ is convex, has boundary values at most those of $v$, and has determinant at least $\det D^2u+R\ge\det D^2v$. The convex Monge--Amp\`ere comparison principle gives $u+b-B_\partial\le v$. Interchanging $u,v$ gives $v+b-B_\partial\le u$. The two inequalities imply the result. The comparison principle is standard; see \citet[Theorem 2.10]{figalli2017} or \citet[Theorem 2.2]{nystrom2023}.
\end{proof}

\section{Finite-unroll derivative recurrences}\label{app:unroll}
The certificate is independent of how its approximate quantities are obtained. For completeness, derivatives of the finite iteration can be computed without symbolic differentiation. Start with $z_0=0$, $J_0=0$ and $K_{ij,0}=0$. At step $\ell$, define
\[
s_\ell=Wz_\ell+Ax+b,\quad d_\ell=\phi'(s_\ell),\quad
\gamma_\ell=\phi''(s_\ell),\quad v_{i,\ell}=A_{:,i}+WJ_{i,\ell}.
\]
Then the chain rule gives
\begin{align*}
z_{\ell+1}&=\phi(s_\ell),\\
J_{i,\ell+1}&=d_\ell\odot v_{i,\ell},\\
K_{ij,\ell+1}&=d_\ell\odot(WK_{ij,\ell})
+\gamma_\ell\odot v_{i,\ell}\odot v_{j,\ell}.
\end{align*}
At depth $L$, take $\hat z=z_L$, $\hat J_i=J_{i,L}$ and $\hat K_{ij}=K_{ij,L}$. The certificate coefficients must be recomputed at $\hat s=Wz_L+Ax+b$; they are not automatically the coefficients from the preceding unroll step. For a finite unroll, the derivative residuals in \eqref{eq:rhofirst}--\eqref{eq:rhosecond} are generally nonzero.

\end{document}